\documentclass{amsart}
\usepackage{amsmath,latexsym,amssymb,amsfonts}
\usepackage{amscd,amssymb,epsfig}
\usepackage{graphicx}

\newtheorem{theorem}{Theorem}[section]
\newtheorem{lemma}[theorem]{Lemma}

\newtheorem{prop}[theorem]{Proposition}

\newtheorem{definition}[theorem]{Definition}
\newtheorem{remark}[theorem]{Remark}

\newcommand\EE{\mathbb{E}}
\newcommand\CC{\mathbb{C}}
\newcommand\cA{\mathcal{A}}
\newcommand\cM{\mathcal{M}}
\newcommand\tr{\text{tr}}
\newcommand\id{\text{id}}
\newcommand\ee{\varepsilon}

\title[Operator-valued free multiplicative convolution]{Operator-valued free multiplicative convolution: analytic subordination theory and applications to random matrix theory}

\author[S. Belinschi]{Serban T. Belinschi}
\address{Department of Mathematics \& Statistics,
Queen's University, and Institute of Mathematics ``Simion
Stoilow'' of the Romanian Academy;
Department of Mathematics and Statistics, 
Queen's University,
Jeffrey Hall, 
Kingston, ON K7L 3N6 CANADA} 
\email{sbelinsch@mast.queensu.ca}

\author[R. Speicher]{Roland Speicher}
\address{Universit\"{a}t des Saarlandes, FR $6.1-$Mathematik, Postfach 151150, 66041 Saarbr\"{u}cken, GERMANY } 
\email{speicher@math.uni-sb.de}

\author[J. Treilhard]{John Treilhard}
\address{Department of Mathematics \& Statistics,
Queen's University;
Department of Mathematics and Statistics, 
Queen's University,
Jeffrey Hall, 
Kingston, ON K7L 3N6 CANADA} 
\email{}

\author[C. Vargas]{Carlos Vargas}
\address{Universit\"{a}t des Saarlandes, FR $6.1-$Mathematik, Postfach 151150, 66041 Saarbr\"{u}cken, GERMANY} 
\email{carlos@math.uni-sb.de}

\date{\today}
\begin{document}

\thanks{Work of S. Belinschi was supported by a Discovery
Grant from NSERC. S. Belinschi also gratefully acknowledges the
support of the Alexander von Humboldt foundation and the 
hospitality of the Free Probability research group at the
Universit\"{a}t des Saarlandes during the work on this paper.\\ 
Work of J. Treilhard and C. Vargas was supported by funds from R. Speicher from the Alfried Krupp von Bohlen und Halbach Stiftung ("R\"uckkehr deutscher Wissenschaftler aus dem Ausland") and from the DFG (SP 419/8-1), respectively.}

\begin{abstract}

We give an explicit description, via analytic subordination, of free multiplicative convolution 
of operator-valued distributions. In particular, the subordination function is obtained from an iteration process. This algorithm is easily numerically implementable. We present two concrete applications of our method: the product of two free operator-valued semicircular elements and
the calculation of the distribution of $dcd+d^2cd^2$ for scalar-valued $c$ and $d$, which are free.  Comparision between the solution obtained by our
methods and simulations of random matrices shows excellent agreement.
\end{abstract}

\maketitle

\section{Introduction}

Free probability is a quite recent theory that has gained interest in 
the last few years. One of its main applications is on the field of random 
matrices. 
More specifically, it provides a conceptual way of understanding the 
distribution of the eigenvalues of several large random matrices. The 
variety 
of matrix models where free probability can be used is growing in 
accordance to the developments of the theory.

The crucial requirement that allows the treatment of a random matrix model 
with the algebraic and analytical machinery of free probability is that 
the matrices involved should satisfy (asymptotically, as their size tends
to infinity) freeness relations. Some of the most important 
random matrices, such as Wigner and Haar Unitary matrices, were shown, starting with the basic work \cite{V1991}, to 
have these freeness requirements among themselves and with respect to 
deterministic matrices.

The applicability of free probability increased rapidly, in 
different directions, with the implementation of 
Voiculescu's operator-valued version of the theory. The main idea is that 
operator-valued freeness is a much less restrictive condition, but still most of the features of usual free probability theory are present. We are now 
able to work with block Gaussian matrices \cite{Sh1996,ROBS}, 
rectangular random matrices of different sizes \cite{BG} and more complicated 
combinations of random and 
deterministic matrices \cite{SV}, where scalar-valued freeness breaks up.

Whereas the analytic theory of scalar-valued free convolutions is far evolved
(and thus we have now a variety of ways to deal with asymptotic eigenvalue distributions of matrices which are free), the same cannot be said about the
status of operator-valued convolutions. In particular, at the moment we do not have an analytic description of operator-valued convolutions which would be easily and controllably implementable on a computer. Thus, numerical investigations for the above mentioned random matrix models were done in a more or less ad hoc way. Our main aim is to improve on
this situation and provide a coherent analytic description of operator-valued
free convolution and show its usefulness for dealing with random matrix
questions. In the present paper we concentrate on multiplicative free convolution, the question of additive free convolution will be addressed in
another paper \cite{BMS}.

The present paper is motivated by the following problem (which was communicated to us by Aris Moustakas in this form in the context of wireless communications and, independently, by Raj Rao for the special case $b_1=b_2$ and $a_1=a_2^2$ in the context of random graphs):
If $\{a_1,a_2\}$ and $\{b_1,b_2\}$ are free, then it is not true in general that the elements $a_1b_1a_1$, $a_2b_2a_2$ are free. This has made the distribution of $c=a_1b_1a_1+a_2b_2a_2$ quite inaccessible up to now.

We observe, however, that the distribution of $c$ is the same (modulo a Dirac mass at zero) as the distribution of the element

\begin{equation} \label{rmprob} \begin{pmatrix} 
  a_1b_1a_1 + a_2b_2a_2 & 0  \\
  0 & 0  
 \end{pmatrix} = \begin{pmatrix}
  a_1 & a_2  \\
  0 & 0
 \end{pmatrix} \begin{pmatrix}
  b_1 & 0  \\
  0 & b_2
 \end{pmatrix}\begin{pmatrix}
  a_1 & 0  \\
  a_2 & 0
 \end{pmatrix},
\end{equation}
which in turn has the same moments as
\begin{equation} \begin{pmatrix}
  a_1^2 & a_1a_2  \\
  a_2a_1 & a_2^2  
 \end{pmatrix} \begin{pmatrix}
  b_1 & 0  \\
  0 & b_2
 \end{pmatrix}=:AB.
\end{equation}

The advantage of this reformulation is that the matrices $A$ and $B$ are free with amalgamation over the algebra 
${M}_2(\mathbb{C})$ of $2\times2$ constant matrices. Hence, the 
distribution that we are looking for will be given by first calculating the
${M}_2(\mathbb{C})$-valued free multiplicative convolution of $A$ and $B$ to obtain the $M_2(\mathbb{C})$-valued distribution of $AB$ and then
getting from this the (scalar-valued) distribution of $AB$ by taking the trace over $M_2(\mathbb{C})$. This has motivated us to look at the general problem of dealing with operator-valued free multiplicative convolutions.

The standard way to deal with free multiplicative convolutions is through 
Voiculescu's $S$-transform. The generalization of the $S$-transform to the 
operator-valued situation was found by Dykema \cite{Str}. (A certain version of the $S$-transform, via a Fock space-type construction, appeared already in the work of Voiculescu \cite{V1995}.) Direct computations of 
operator-valued $S$-transforms for non-trivial elements are, however, 
extremely hard. Approximations can be done on domains corresponding to 
domains of the Cauchy-Stieltjes transform which are located far away from 
the real axis. For practical purposes this is a problem, since we are 
interested in the behavior of the Cauchy transform close to the real axis, 
in order to recover the distribution by an Stieltjes inversion.

In this paper we follow Biane's approach \cite{Biane1} to scalar-valued 
free 
multiplicative convolution via analytic subordination, which was later 
extended by Voiculescu \cite{V2000,FreeMarkov} to the operator valued case. 

Our main contribution to the theory is to find the subordination 
functions as iterative 
limits, similar to what has been done in the scalar case \cite{BB07}. We 
rely on the twisted multiplicative property of Dykema's $S$-transform.

This will allow us to set up fixed point equations to approximate 
effectively the values of the Cauchy transform of $AB$ in the whole 
operator-valued 
upper half plane and in particular, close to the real axis. As inputs, we 
require the individual operator-valued Cauchy Transforms of $A$ and $B$ (or good approximations of these).

We want to stress that in the operator-valued context there are rarely situations where one has an analytic formula for the involved Cauchy transforms. The best one can hope for are equations which determine those transforms. In order to be applicable to concrete problems one needs of course a way to solve these equations - in particular, to single out the correct solution; our operator-valued equations usually have many solutions, only one of them corresponding to the wanted Cauchy transform. Since the characterization of the Cauchy transform among all solutions is by positivity requirements this is an intrinsic analytic characterization, which indicates that those equations cannot be solved by pure formal power series expansion arguments. So what we need for a theory of operator-valued convolution which is also practically applicable is to set up our theory in such a form which can also numerically be implemented (and such that the theory provides arguments for the working of this implementation). What has become more and more apparent in the scalar-valued context, namely that the subordination formulation of free convolution seems to be the right choice, is even more prominent in the operator-valued context. Trying to solve an operator-valued free multiplicative convolution problem directly with the help of the operator-valued $S$-transform becomes quite a challenging task very soon, whereas using the subordination formulation, as presented in 
this paper, allows not only a satisfying analytic description of the involved
transformations, but can also be implemented numerically very easily. 

In Section 2, we will present our analytic description, via subordination, of the free multiplicative operator-valued convolution. In Section 3, we will show
the usefulness of our approach by implementing our method to calculate the distribution of the product of two operator-valued free random variables and by comparing this with histograms for corresponding random models.
We consider there two different types
of examples: (i) the product of two operator-valued semicircular elements, the individual 
Cauchy transforms obtained by the fixed point equations described in 
\cite{HRS07}; (ii) in the context of our original problem described above, we treat the case $aba+a^2ba^2$, where $a$ is discrete and $b$ is either discrete or a semicircular variable. 

\section{Multiplication of operator-valued free random variables}

We will call an {\em operator-valued non-commutative probability space}
a triple $(\mathcal M,\mathbb E,B)$, where $\mathcal M$ is a von 
Neumann algebra, $B\subseteq\mathcal M$ is a $W^*$-subalgebra
containing the unit of $\mathcal M$, and $\mathbb E\colon\mathcal M
\to B$ is a unit-preserving conditional expectation. Elements
in $\mathcal M$ will be called operator-valued (or $B$-valued)
random variables. The distribution of a random variable
$x\in\mathcal M$ with respect to $\mathbb E$ is, by definition,
the set of multilinear maps
$$
\mu_x:=\{m_n^x\colon B^{n-1}\to B\colon m_n(b_1,\dots,b_{n-1})=
\mathbb E[xb_1xb_2\cdots xb_{n-1}x], n\in\mathbb N\}.
$$
We call $m_n^x$ the $n^{\rm th}$ moment of $x$ (or, equivalently,
of $\mu_x$).
It will be convenient to interprete $m_0^x$ as the constant
equal to $1$, the unit of $B$ (or, equivalently, of $\mathcal M$)
and $m_1^x=\mathbb E[x]$, the expectation of $x$.
We denote by $B\langle x_1\dots,x_n\rangle$ the ${}^*$-algebra
generated by $B$ and the elements $x_1,\dots,x_n$ in $\mathcal M$.

\begin{definition}\label{B-freeness}
Two algebras $A_1,A_2\subseteq\mathcal M$ containing $B$ are called 
{\em free with amalgamation over} $B$ with respect to $\mathbb E$
(or just {\em free over} $B$) if
$$
\mathbb E[x_1x_2\cdots x_n]=0
$$
whenever $n\in\mathbb N$, $x_j\in A_{i_j}$ satisfy $\mathbb E[x_j]=0$
and $i_j\neq i_{j+1}$, $1\leq j\leq n-1$. Two random variables $x,y\in
\mathcal M$ are called free over $B$ if $B\langle x\rangle$ and
$B\langle y\rangle$ are free over $B$.
\end{definition}
If $x,y\in\mathcal M$ are free over $B$, then $\mu_{x+y}$ and
$\mu_{xy}$ depend only on $\mu_x$ and $\mu_y$. Following 
Voiculescu, we shall denote these
dependencies by $\mu_x\boxplus\mu_y$ and $\mu_x\boxtimes\mu_y$, 
and call them the free additive, respectively free multiplicative,
convolution of the distributions $\mu_x$ and $\mu_y$.
It is known \cite{RS2,V1995} that both $\boxplus$ and $\boxtimes$
are associative, but $\boxtimes$ may fail to be commutative.

\subsection{Analytic transforms}

A very powerful tool for the study of operator-valued 
distributions is the generalized Cauchy-Stieltjes transform and
its fully matricial extension \cite{V1995,V2000}: for a fixed
$x\in\mathcal M$, we define $G_x(b)=\mathbb E\left[(b-x)^{-1}\right]$
for all $b\in B$ so that $b-x$ is invertible in $\mathcal M$. One can 
easily verify that $G_x$ is a holomorphic mapping on an open subset of 
$B$. Its {\em fully matricial extension} $G_x^{(n)}$ is defined on
the set of elements $b\in M_n(B)$ for which $b-x\otimes 1_n$ is invertible in $M_n(\mathcal M)$, by the relation $G_x^{(n)}(b)=
\mathbb E\otimes\text{Id}_{M_n(\mathbb C)}\left[(b-x\otimes 1_n)^{-1}
\right]$. It is a crucial observation of Voiculescu that 
the family $\{G_x^{(n)}\}_{n\ge1}$ encodes the distribution
$\mu_x$ of $x$. A succint description of how to identify
the $n^{\rm th}$ moment of $x$ when $\{G_x^{(n)}\}_{n\ge1}$ is
known is given in \cite{BPV2010}.

In the following we will use the notation $x>0$ for the situation where
$x\geq 0$ and $x$ is invertible; note that this is equivalent to the fact
that there exists a real $\ee\geq 0$ such that $x\geq \ee 1$. From the later
it is clear that $x>0$ implies $\EE[x]>0$ (because our conditional expectations are automatically completely positive).

From now on we shall restrict our attention to the case when
$x,y$ are selfadjoint, and (for most applications) nonnegative. 
In this case, one of appropriate domains for $G_x$ - and the 
domain we will use most - is the {\em operator upper half-plane}
$\mathbb H^+(B):=\{b\in B\colon\Im b>0\}.$ Elements in this open set
are all invertible, and $\mathbb H^+(B)$ is invariant under conjugation
by invertible elements in $B$. It has been noted
in \cite{V2000} that $G_x^{(n)}$ maps $\mathbb H^+(M_n(B))$ into
the operatorial lower half-plane $\mathbb H^-(M_n(B)):=-
\mathbb H^+(M_n(B))$ and has ``good behaviour at infinity'' in the
sense that $\displaystyle\lim_{\|b^{-1}\|\to0}bG_x^{(n)}(b)=
\lim_{\|b^{-1}\|\to0}G_x^{(n)}(b)b=1$. 

As, from an analytic function perspective, $G_x^{(n)}$ have 
essentially the same behaviour on $\mathbb H^+(M_n(B))$ for any
$n\in\mathbb N$, we shall restrict our analysis from now on to 
$G_x:=G_x^{(1)}$. However, all properties we deduce for this
$G_x$, and all the related functions we shall introduce, remain
true, under the appropriate formulation, for all $n\ge1$.

We shall use the following analytic mappings, all defined on
$\mathbb H^+(B)$; all transforms have a natural
Schwarz-type analytic extension to the lower half-plane given by
$f(b^*)=f(b)^*$;
in all formulas below, $x=x^*$ is fixed in $
\mathcal M$:
\begin{itemize}
\item the moment generating function:
\begin{equation}\label{psi}
\Psi_x(b)=\mathbb E\left[(1-bx)^{-1}-1\right]=\mathbb E\left[
(b^{-1}-x)^{-1}\right]b^{-1}-1=G_x(b^{-1})b^{-1}-1;
\end{equation}
\item The reciprocal Cauchy transform:
\begin{equation}\label{F}
F_x(b)=\mathbb E\left[(b-x)^{-1}\right]^{-1}=G_x(b)^{-1};
\end{equation}
\item The eta transform (Boolean cumulant series):
\begin{equation}\label{eta}
\eta_x(b)=\Psi_x(b)\left(1+\Psi_x(b)\right)^{-1}=1-b
\mathbb E\left[(b^{-1}-x)^{-1}\right]^{-1}=1-bF_x(b^{-1});
\end{equation}
\item In this paper we shall call this function ``the h transform:''
\begin{equation}\label{h}
h_x(b)=b^{-1}\eta_x(b)=b^{-1}-\mathbb E\left[(b^{-1}-x)^{-1}\right]^{-1}=b^{-1}-F_x(b^{-1});
\end{equation}
\end{itemize}
Based on the moment generating function, Dykema \cite{Str}
introduced the operator-valued version of Voiculescu's $S$-transform
\cite{V2} as an analytic mapping on Banach algebras (an earlier,
less easily employed version can be found in \cite{V1995}).
It is easy to note that $\Psi_x'(b)(c)=\mathbb E\left[
(1-bx)^{-1}cx(1-bx)^{-1}\right]$, so that 
$\Psi_x'(0)(c)=\mathbb E[cx]=c\mathbb E[x]$. Under the assumption
that $\mathbb E[x]$ is invertible in $B$ (and, in particular, when $x>0
$), the linear map
$\Psi'_x(0)$ becomes invertible, with inverse $c\mapsto c\mathbb E
[x]^{-1}$, and so by the usual Banach-space inverse function theorem,
$\Psi_x$ has an inverse around zero, which we shall denote by
$\Psi_x^{-1}$. The $S$-transform is defined as
\begin{equation}\label{S}
S_x(b)=b^{-1}(1+b)\Psi_x^{-1}(b),\quad \|b\|\text{ small enough.}
\end{equation}
Dykema showed \cite[Theorem 1.1]{Str} that, whenever $\mathbb E[x]$
and $\mathbb E[y]$ are both invertible in $B$, 
\begin{equation}\label{Str-property}
S_{xy}(b)=S_y(b)S_x(S_y(b)^{-1}bS_y(b)),\quad\|b\|\text{ small enough.}
\end{equation}

\subsection{Three ways to the subordination function}

Here we shall describe three ways to finding the analytic 
subordination functions for free multiplicative convolution of 
operator-valued distributions. The analytic subordination has
been proved in different contexts by Voiculescu and Biane
\cite{V3,Biane1,V2000,FreeMarkov}. For the case $B=\mathbb C$,
Biane showed that there exist analytic functions 
$\omega_1,\omega_2\colon\mathbb C\setminus[0,+\infty)\to
\mathbb C\setminus[0,+\infty)$ which preserve half-planes and
satisfy $\mathbb E_{\mathbb C\langle y\rangle}\left[(1-zxy)^{-1}
\right]=(1-\omega_2(z)y)^{-1},$ $\mathbb E_{\mathbb C\langle x\rangle}
\left[(1-zxy)^{-1}\right]=(1-\omega_1(z)x)^{-1}$. In particular,
$\eta_x(\omega_1(z))=\eta_y(\omega_2(z))=\eta_{xy}(z),$
$z\in \mathbb C\setminus[0,+\infty)$. Voiculescu extended 
this relation, essentially in \cite{V2000}, and made it more
precise in \cite{FreeMarkov}, to the case of a general $B$,
under the assumption that $\mathcal M$ is endowed with a tracial 
state and $\mathbb E$ preserves this trace.
Later, in \cite{BB07}, Biane's subordination functions $\omega_1,
\omega_2$ were found as limits of an iteration process involving
$\eta_x$ and $\eta_y$. A precise description of such an 
iterative process for (very general) positive operator-valued 
random variables, in the spirit of \cite{HRS07}, will be our
main contribution in this section.

Inspired by the shape of formula \eqref{Str-property}, we claim that
for our purposes, the most appropriate writing of the operator-valued
subordination phenomenon is the following:
\begin{theorem}\label{Main-op}
Let $x>0,y=y^*\in\mathcal M$ be two random 
variables with invertible expectations, free over $B$. 
There exists a G\^{a}teaux holomorphic map $
\omega_2\colon\{b\in B\colon\Im(bx)>0\}\to\mathbb H^+(B),$
such that
\begin{enumerate}
\item $\eta_y(\omega_2(b))=\eta_{xy}(b)$, $\Im(bx)>0$;
\item $\omega_2(b)$ and $b^{-1}\omega_2(b)$ are analytic around zero;
\item For any $b\in B$ so that $\Im(bx)>0$, the map $g_b\colon
\mathbb H^+(B)\to\mathbb H^+(B)$, $g_b(w)=bh_x(h_y(w)b)$ is well-defined, analytic and 
$$
\omega_2(b)=\lim_{n\to\infty}g_b^{\circ n}(w),
$$
for any fixed $w\in\mathbb H^+(B)$.
\end{enumerate}
Moreover, if one defines $\omega_1(b):=h_y(\omega_2(b))b$, then
$$
\eta_{xy}(b)=\omega_2(b)\eta_x(\omega_1(b))\omega_2^{-1}(b),
\quad\Im (bx)>0.
$$
\end{theorem}
Following the proof of this theorem, we shall mention several
conditions under which some hypotheses, particularly the - rather
inconvenient in practical applications - invertibility requirement, 
can be weakened or entirely dropped.

\begin{proof}
We shall split our proof in several remarks, formulas and lemmas.
For our purposes, a slight variation of the $S$-transform will be
more useful: we define the sigma transform $\Sigma_x(b)=b^{-1}
\eta_x^{-1}(b)$, again on a neighbourhood of zero. Elementary 
arithmetic manipulations
show that $\Sigma_x(b)=S_x(b(1-b)^{-1})$ and so $\Sigma_{xy}(b)=
\Sigma_y(b)\Sigma_x(\Sigma_y(b)^{-1}b\Sigma_y(b))$.
Using this relation we write on a neighbourhood of zero
in $B$:
\begin{eqnarray*}
b\eta_{xy}^{-1}(b)& = & b^2\Sigma_{xy}(b)\\
& = & b^2\Sigma_y(b)\Sigma_x(\Sigma_y(b)^{-1}b\Sigma_y(b))\\
& = & b^2\Sigma_y(b)(\Sigma_y(b)^{-1}b\Sigma_y(b))^{-1}\eta_x^{-1}(\Sigma_y(b)^{-1}b\Sigma_y(b))\\
& = & b\Sigma_y(b)\eta_x^{-1}(\Sigma_y(b)^{-1}b\Sigma_y(b))\\
& = & \eta_{y}^{-1}(b)\eta_x^{-1}\bigl((\eta_{y}^{-1}(b))^{-1}b
\eta_{y}^{-1}(b)\bigr).
\end{eqnarray*}
Now define $\omega_2(b)=\eta_y^{-1}(\eta_{xy}(b))$, again for $\|b\|$
sufficiently small.
We substitute in the previous relation
$\eta_{xy}(b)$ for $b$ to obtain
\begin{eqnarray*}
\eta_y(\omega_2(b))b & = & \eta_{xy}(b)b\\
&=&\eta_{y}^{-1}(\eta_{xy}(b))\eta_x^{-1}((\eta_{y}^{-1}(\eta_{xy}(b)))^{-1}\eta_{xy}(b)
\eta_{y}^{-1}(\eta_{xy}(b)))\\
&=&\omega_2(b)\eta_{x}^{-1}(\omega_2(b)^{-1}\eta_y(\omega_2(b))
\omega_2(b)).
\end{eqnarray*}
Recalling from equation \eqref{h} the definition of the h 
transform, we obtain
$$
h_y(\omega_2(b))b=\eta_x^{-1}(h_y(\omega_2(b))\omega_2(b)), 
\quad\|b\|\text{ small enough}.
$$
An application of $\eta_x$ on both sides gives $\eta_x(h_y(\omega_2(b))
b)=h_y(\omega_2(b))\omega_2(b)$, or 
\begin{equation}\label{7}
\omega_2(b)=bh_x(h_y(\omega_2(b))b)=g_b(\omega_2(b)),
\quad\|b\|\text{ small enough}.
\end{equation}
This shows us that $\omega_2$ exists on a small enough
neighbourhood of the origin, and is a fixed point for the
map $g_b$ introduced in Theorem \ref{Main-op}. Moreover,
this indicates that $b^{-1}\omega_2(b)$ is
analytic around zero (a fact that follows quite easily
also from the definition of $\omega_2$ as $\eta_y^{-1}\circ
\eta_{xy}$). The most significant, however, is the fact that,
under the conditions of analyticity of the two maps
$\omega_2(b)$ and $b^{-1}\omega_2(b)$ around zero, equation
\eqref{7} uniquely determines $\omega_2$, and thus a function
satisfying \eqref{7} must also satisfy $\eta_y\circ\omega_2=
\eta_{xy}$, and vice-versa.

Now we observe that, under the additional assumption of the existence
of a trace $\tau$ on $\mathcal M$ so that $\tau=\tau\circ\mathbb E$,
this function coincides with a function provided by Voiculescu: in 
\cite{V2000}, Voiculescu proved that, whenever $x$ and $y$ are free
over $B$, the range of the analytic map $b\mapsto y+\mathbb E_{B\langle 
y\rangle}\left[(b-x-y)^{-1}\right]^{-1}$ is included in $B$.
($\mathbb  E_{B\langle y\rangle}$ denotes the conditional expectation with $\tau$ onto the von Neumann algebra generated by $B$ and $y$.)
It follows quite easily that the range of the map
$b\mapsto y^{-1}-\mathbb E_{B\langle 
y\rangle}\left[(y^{-1}-bx)^{-1}\right]^{-1}$ is also included in $B$.
We claim that 
$$
\omega_2(b)=y^{-1}-\mathbb E_{B\langle 
y\rangle}\left[(y^{-1}-bx)^{-1}\right]^{-1}.
$$
Since $\omega_2$ up to this moment has only been defined on
a neighbourhood of zero, we only need to verify the
equality for $\|b\|$ small. Indeed, the above relation
is equivalent to $\mathbb E_{B\langle 
y\rangle}\left[(1-bxy)^{-1}\right]^{-1}=1-\omega_2(b)y$, by the 
bimodule property of $\mathbb E_{B\langle 
y\rangle}$. Inverting both sides of the equality and applying
the conditional expectation $\mathbb E$ gives 
$\mathbb E\left[(1-bxy)^{-1}\right]=\mathbb E\left[(1-\omega_2(b)y)^{-1}\right].$ We invert, take a $b$ as a factor and subtract 1
in order to get 
$\eta_y(\omega_2(b))=\eta_{xy}(b)$. As it was noted by Voiculescu
\cite{V2000}, and we shall argue below, $y^{-1}-bx$ is invertible 
whenever 
$\Im (bx)>0$. Since the set $\{b\in B\colon\Im(bx)>0\}\cap
\{b\in B\colon\|b\|<\varepsilon\}$ is open for any $\varepsilon
>0$, and $\displaystyle\lim_{\|b\|\to0}\left(y^{-1}-\mathbb E_{B\langle 
y\rangle}\left[(y^{-1}-bx)^{-1}\right]^{-1}\right)=0$, it follows,
by analytic continuation, that there exists an $\varepsilon>0$
so that 
\begin{eqnarray}
\lefteqn{\omega_2\colon\{b\in B\colon\Im(bx)>0\}\cup
\{b\in B\colon\|b\|<\varepsilon\}  \to  B}\nonumber\\ 
& & \omega_2(b)  = y^{-1}-\mathbb E_{B\langle 
y\rangle}\left[(y^{-1}-bx)^{-1}\right]^{-1}.
\end{eqnarray}

Until now we have argued that an analytic map $\omega_2$
satisfying parts (1) and (2) of our theorem exists and is unique
(i) on a neighbourhood of zero for any $B$-valued non-commutative
probability space $(\mathcal M,\mathbb E,B)$, and (ii) that this
$\omega_2$ extends analytically to $\{b\in B\colon\Im(bx)>0\}$
when $\mathcal M$ is endowed with a tracial state $\tau$ so that
$\tau\circ\mathbb E$ remains a trace.
We go now to the third way of identifying $\omega_2$,
namely as a fixed point of an analytic mapping, method which
will allow us to extend $\omega_2$ to a set of the form
$\{b\in B\colon\Im(bx)>0\}\cup
\{b\in B\colon\|b\|<\varepsilon\}$ for any type of 
non-commutative
probability space $(\mathcal M,\mathbb E,B)$.
\begin{lemma}\label{domain}
Assume that $x\ge0$ is invertible in $\mathcal M$ and
$b\in\{b\in B\colon\Im(bx)>0\}$. If $c\in
B$ is so that $\Im c\ge0$, then $\Im h_{bx}(c)\ge\mathbb E\left[(\Im(bx))^{-1}\right]^{-1}>0$.
\end{lemma}
\begin{proof}
For simplicity, let us replace $c\in\mathbb H^+(B)$ by 
$-c^{-1}\in\mathbb H^+(B)$, so that
$$
h_{bx}(-c^{-1})=\mathbb E\left[(c+bx)^{-1}\right]^{-1}-c.
$$
We shall split our problem in two and use the same method as in \cite{BPV2010}: assume that
$\varphi$ is an arbitrary positive linear functional on $B$
so that $\varphi(1)=1$. We define
$$
f_\varphi\colon\mathbb C^+\to\mathbb C^+,\quad f_\varphi(z)=\varphi
\left(\mathbb E\left[(\Re(c+bx)+z\Im(c+bx))^{-1}\right]^{-1}\right).
$$
As $\Im c\ge0$, and $\Im(bx)$, $\Im z$ are all strictly positive, it follows
that $\Im(\Re(c+bx)+z\Im(c+bx))^{-1}<0$, and, since $\mathbb E$ is
positive and faithful, $\Im\mathbb E
\left[(\Re(c+bx)+z\Im(c+bx))^{-1}\right]<0$. Thus,
$\Im\left(\mathbb E
\left[(\Re(c+bx)+z\Im(c+bx))^{-1}\right]^{-1}\right)>0$, and, in
particular, invertible. Since $\varphi\ge0,\varphi(1)=1$, the
Cauchy-Schwarz inequality tells us that $\Im f_\varphi(z)>0$
whenever $\Im z>0$. We take 
\begin{eqnarray*}
\lim_{z\to\infty}\frac{f_\varphi(z)}{z} & = & 
\varphi\left(\lim_{z\to\infty}\mathbb E
\left[\left(\frac{\Re(c+bx)}{z}+\Im(c+bx)\right)^{-1}\right]^{-1}\right)\\
& = & \varphi\left(\mathbb E
\left[(\Im(c+bx))^{-1}\right]^{-1}\right).
\end{eqnarray*}
The Nevanlinna representation \cite[Chapter III]{akhieser}
allows us to write the inequality $\Im f_\varphi(z)\ge\varphi\left(\mathbb E
\left[(\Im(c+bx))^{-1}\right]^{-1}\right)\Im z$ for all
$z\in\mathbb C^+$. Equality holds at one point of $\mathbb C^+$ 
if and only if $f_\varphi
(z)=\varphi\left(\mathbb E
\left[(\Im(c+bx))^{-1}\right]^{-1}\right)z+\text{real constant}$.
Taking $z=i$ in the above we get
$$\Im\varphi\bigl(\mathbb E
\left[(c+bx)^{-1}\right]^{-1}\bigr)\ge\varphi\bigl(\mathbb E
\left[(\Im(c+bx))^{-1}\right]^{-1}\bigr).$$ 
Since this inequality holds for all states $\varphi$ on $B$,
we conclude that 
$$\Im\mathbb E
\left[(c+bx)^{-1}\right]^{-1}\ge\mathbb E
\left[(\Im(c+bx))^{-1}\right]^{-1}.$$ 
We shall now prove that
$\mathbb E
\left[(\Im(c+bx))^{-1}\right]^{-1}>\Im c$:
\begin{eqnarray}
\lefteqn{\mathbb E
\left[(\Im(c+bx))^{-1}\right]^{-1}>\Im c \iff
\mathbb E\nonumber
\left[(\Im c+\Im(bx))^{-1}\right]<(\Im c)^{-1}}\\
& \iff & \mathbb E\nonumber
\left[\sqrt{\Im c}(\Im c+\Im(bx))^{-1}\sqrt{\Im c}\right]<1\\
& \iff & \mathbb E\nonumber
\left[\left(1+\left(\sqrt{\Im c}\right)^{-1}\Im(bx)\left(\sqrt{\Im c}\right)^{-1}\right)^{-1}\right]<1\\
& \Longleftarrow & \left(1+\left(\sqrt{\Im c}\right)^{-1}\Im(bx)\left(\sqrt{\Im c}\right)^{-1}\right)^{-1}<1.\nonumber
\end{eqnarray}
The last inequality is trivially true by functional calculus and the invertibility of $\Im(bx)$. Putting the inequalities together gives
$$\varphi\bigl(\Im \mathbb E
\left[(c+bx)^{-1}\right]^{-1}\bigr)=
\Im\varphi\bigl(\mathbb E
\left[(c+bx)^{-1}\right]^{-1}\bigr)>\varphi(\Im c).$$
Since this is true for all positive $\varphi$, we get that
$h_{bx}$ maps $\mathbb H^+(B)$ into itself.

Next, we make use again of the same trick: we define 
$$
f_\varphi\colon\mathbb C^+\to\mathbb C^+,\quad f_\varphi(z)=\varphi\left(\mathbb E
\left[(c+\Re(bx)+z\Im(bx))^{-1}\right]^{-1}-c\right).
$$
As before, whenever $\Im c\ge0$ (the case $c=c^*$, for example, is not
excluded here), $\Im\mathbb E
\left[(c+\Re(bx)+z\Im(bx))^{-1}\right]^{-1}-\Im c>0$ for any $z\in\mathbb C^+$. We take again limit as $z\to\infty$ of 
$f_\varphi(z)/z$ to obtain $\varphi\left(\mathbb E\left[(\Im(bx))^{-1}\right]^{-1}\right)$. The same argument used above implies that
$\Im f_\varphi(z)\ge\varphi(\mathbb E\left[(\Im(bx))^{-1}\right]^{-1})\Im z$, and so, for $z=i$ we obtain
$\Im\varphi\left(\mathbb E
\left[(c+\Re(bx)+i\Im(bx))^{-1}\right]^{-1}-c\right)\ge
\varphi(\mathbb E\left[(\Im(bx))^{-1}\right]^{-1})$, for all $\varphi
\ge0,\varphi(1)=1$. This implies
$\Im h_{bx}(c)\ge\mathbb E\left[(\Im(bx))^{-1}\right]^{-1}$, as claimed.
\end{proof}

By the definition of the h transform \eqref{h}, we have 
$bh_x(cb)=h_{bx}(c)$. 
The previous lemma allows us to write
$$
\Im h_y(w)\ge0\implies\Im h_{bx}(h_y(w))\ge
\mathbb E\left[(\Im(bx))^{-1}\right]^{-1}>0,
$$
i.e. $\Im g_b(w)\ge\mathbb E\left[(\Im(bx))^{-1}\right]^{-1}>0$ for any $w\in\mathbb H^+(B)$, and $b$ with $\Im(bx)>0$.
Thus, we have shown that $g_b$ lands strictly into $\mathbb H^+(B)$
whenever $\Im(bx)>0$ and $\Im h_y(w)\ge0.$

\begin{remark}\label{pos}
It has been shown in \cite{BPV2010} that 
$\Im \mathbb E\left[(w-y)^{-1}\right]^{-1}\ge\Im w$ for any
$w\in\mathbb H^+(B)$ and $y=y^*\in\mathcal M$. By noting that
$-w^{-1}\in\mathbb H^+(B)$ if and only if $w\in\mathbb H^+(B)$,
we obtain that $\Im h_y(w)\ge0$ and $\Im h_y(-w^{-1})\ge0$
for any $w\in\mathbb H^+(B)$.
\end{remark}

We can now improve on our previous statement: for $x>0,y=y^*$,
and $b\in B$ so that $\Im (bx)>0$,
$$
\Im w>0\implies\Im h_y(w)\ge0\implies\Im h_{bx}(h_y(w))\ge
\mathbb E\left[(\Im(bx))^{-1}\right]^{-1}>0,
$$
i.e. $\Im g_b(w)\ge\mathbb E\left[(\Im(bx))^{-1}\right]^{-1}>0$ for any $w\in\mathbb H^+(B)$, and $b$ with $\Im(bx)>0$.

\begin{remark}
The functions h associated to selfadjoints $y\in\mathcal M$ have convergent power
series expansions around the origin. Indeed,
\begin{eqnarray*}
\lefteqn{\mathbb E\left[1-wy\right]-\mathbb E\left[(1-wy)^{-1}\right]^{-1}}\\
& = & \mathbb E\left[1-wy\right]\left(\mathbb E\left[(1-wy)^{-1}\right]
-E\left[1-wy\right]^{-1}\right)\mathbb E\left[(1-wy)^{-1}\right]^{-1}\\
& = & \mathbb E\left[1-wy\right]w\left(\sum_{n=1}^\infty\mathbb E
\left[y(wy)^n\right]-\mathbb E[y]\mathbb E[wy]^n
\right)\mathbb E\left[(1-wy)^{-1}\right]^{-1}.
\end{eqnarray*}
Thus, 
\begin{eqnarray*}
h_y(w)&=&w^{-1}\left(1-\mathbb E\left[(1-wy)^{-1}\right]^{-1}\right)\\
& = & w^{-1}\left[\mathbb E\left[1-wy\right]w\left(\sum_{n=1}^\infty\mathbb E
\left[y(wy)^n\right]-\mathbb E[y]\mathbb E[wy]^n
\right)\mathbb E\left[(1-wy)^{-1}\right]^{-1}\right]\\
& & \mbox{}+w^{-1}\mathbb E[wy]\\
&= & (1-\mathbb E[yw])\left[\left(\sum_{n=1}^\infty\mathbb E
\left[y(wy)^n\right]-\mathbb E[y]\mathbb E[wy]^n
\right)\mathbb E\left[(1-wy)^{-1}\right]^{-1}\right]+\mathbb E[y]
\end{eqnarray*}
which gives the power series espansion of $h_y$ around zero 
and shows that $h_y(0)=\mathbb E[y]$.
\end{remark}

We note that, as shown in the above Remark,
there exists a $\varepsilon>0$ so that $h_x,h_y$ are defined on
$\{w\in B\colon\|w\|<\varepsilon\}$ and 
$$
\max\{\|h_x(w)\|,\|h_y(w)\|\colon w\in B,\|w\|<\varepsilon\}
<2(\|x\|+\|y\|).
$$
Choosing $b\in B$ with $\|b\|<(\varepsilon/41)\cdot(\|x\|+\|y\|)^{-1}$
guarantees that 
$\|h_y(w)b\|<\varepsilon/10$ and $\|bh_x(h_y(w)b)\|<\varepsilon/2$,
so $g_b$ maps $\{w\in B\colon\|w\|<\varepsilon\}$ into
$\{w\in B\colon\|w\|<\varepsilon/2\}$. Thus, by the Earle-Hamilton
Theorem \cite[Theorem 11.1]{Din}, $g_b$ has a unique attracting fixed point in
$\{w\in B\colon\|w\|<\varepsilon\}$ which we shall denote by
$\omega_2(b)$, and 
$$
\lim_{n\to\infty}g_b^{\circ n}(w):=\omega_2(b)
$$
exists for all $w\in\{w\in B\colon\|w\|<\varepsilon\}$,
$\|b\|<(\varepsilon/41)\cdot(\|x\|+\|y\|)^{-1}$. The correspondence
$b\mapsto\omega_2(b)$ is clearly analytic, being a uniform limit
of analytic maps. Moreover, on a small enough neighbourhood of
zero, $\eta_y(\omega_2(b))=\eta_{xy}(b)$ by the uniqueness of the
fixed point guaranteed by the Earle-Hamilton Theorem.

For our fixed $x,y\in\mathcal M$ given in the statement of our theorem,
let us fix an $\varepsilon>0$ as above. Consider the set
$$
\left\{b\in B\colon\|b\|<\frac{\varepsilon}{41(\|x\|+\|y\|)}
\right\}\cup\left\{b\in B\colon
\Im (bx)>\frac{\varepsilon}{99\|x^{-1}\|(\|x\|+\|y\|)}\right\}.
$$
This set is clearly open and connected (we can find 
the element $b=i\frac{\varepsilon}{82\|x^{-1}\|(\|x\|+\|y\|)}$
in both open sets whose union we considered).
The above indicates that

\begin{enumerate}
\item $g
\colon\left\{b\in B\colon\|b\|<\frac{\varepsilon}{41(\|x\|+\|y\|)}
\right\}\times\left\{w\in B\colon\|w\|<\varepsilon\right\}\to
\left\{w\in B\colon\|w\|<\frac{\varepsilon}{2}
\right\}$ given by $g_b(w)=h_{bx}(h_y(w))$ is analytic;
\item
$g\colon\left\{b\in B\colon
\Im (bx)>\frac{\varepsilon}{99\|x^{-1}\|(\|x\|+\|y\|)}\right\}
\times\left\{w\in B\colon\Im w>\frac{\varepsilon}{200\|x^{-1}\|(\|x\|+\|y\|)}\right\}$
$\to
\left\{w\in B\colon\Im w>\mathbb E[(\Im(bx))^{-1}]^{-1}
\right\}$ given by $g_b(w)=h_{bx}(h_y(w))$ is analytic.
\end{enumerate}
The union of the two domains of $g$ mentioned above is again
connected, as we immediately note by identifying the point
$(b,w)=\left(\frac{i\varepsilon}{82\|x^{-1}\|(\|x\|+\|y\|)},\frac{i
\varepsilon}{2}\right)$ in both sets. Since the following chain
of implications holds,
$$
\Im(bx)>\kappa\implies(\Im(bx))^{-1}<\frac1\kappa\implies\mathbb E
[(\Im(bx))^{-1}]<\frac1\kappa\implies\mathbb E
[(\Im(bx))^{-1}]^{-1}>\kappa,
$$ 
the {\em strict} inclusion
$$
\left\{w\in B\colon\Im w>\mathbb E[(\Im(bx))^{-1}]^{-1}
\right\}\subset\left\{w\in B\colon\Im w>\frac{\varepsilon}{200\|x^{-1}\|(\|x\|+\|y\|)}\right\}
$$
holds whenever $\Im(bx)>\frac{\varepsilon}{99\|x^{-1}\|(\|x\|+\|y\|)}$.

Next, for any positive linear functional $\varphi$
on $B$ and $b\in B$ so that $\Im (bx)>\frac{\varepsilon}{99\|x^{-1}\|(\|x\|+
\|y\|)}$, we shall show that $\{\varphi(g_b^{\circ n}(w))
\}_{n\in\mathbb N}$ is bounded. The argument uses the fact that the
set $\{b\in B\colon\Im(bx)>0\}$ is convex. we choose $b_1\in B$, 
$\Im (b_1x)>\frac{\varepsilon}{99\|x^{-1}\|(\|x\|+\|y\|)}$,
$\|b_1\|<\frac{\varepsilon}{41(\|x\|+\|y\|)}$ and 
consider $t\mapsto\varphi(g_{tb+(1-t)b_1}^{\circ n}(w))$. The map
$[0,1]\ni t\mapsto g_{tb+(1-t)b_1}^{\circ n}(w)$ lands, for a 
fixed $w\in\mathbb H^+(B)$, entirely in $\mathbb H^+(B)+i
\frac{\varepsilon}{99\|x^{-1}\|(\|x\|+
\|y\|)}$, independently of $n\in\mathbb N$. Thus, there exists
a small enough simply connected complex neighbourhood
$V$ of $[0,1]$, which does {\em not} depend on $n$ so that
$V\ni t\mapsto g_{tb+(1-t)b_1}^{\circ n}(w)$ still lands in
$\mathbb H^+(B)$ for all $n\in\mathbb N$.
We obtain that all maps in the family 
$$
\left\{V\ni t\mapsto\varphi(g_{tb+(1-t)b_1}^{\circ n}(w))\right\}_{
n\in\mathbb N}
$$
take values in $\overline{\mathbb C^+}$. This means that the
family is normal (as a family of functions between complex domains). 
On the other hand, for 
$|t|$ very small, we know that 
$$
\lim_{n\to\infty}g_{tb+(1-t)b_1}^{\circ n}(w)=
\omega_2(tb+(1-t)b_1)\in B,
$$
which, in particular, means that $\lim_{n\to\infty}\varphi(
g_{tb+(1-t)b_1}^{\circ n}(w))=\varphi(\omega_2(tb+(1-t)b_1))$ exists
and is finite. This, together with the above argued normality, implies
that $\lim_{n\to\infty}
\varphi(g_{tb+(1-t)b_1}^{\circ n}(w))$ exists as a holomorphic function
from the given neighbourhood $V$ of $[0,1]$ to $\mathbb C^+$.
Now any linear functional on $B$ has a unique Jordan decomposition
as a linear combination of four positive linear functionals.
Thus, for any $\varphi$ in the dual of $B$, the family
$\{|\varphi(g_{tb+(1-t)b_1}^{\circ n}(w))|\colon t\in V,n\in\mathbb N
\}$ is bounded. The uniform boundedness principle
(see, for example, \cite[Lemma 1]{AT}) guarantees that
$\{\|g_{tb+(1-t)b_1}^{\circ n}(w)\|\colon t\in V,n\in\mathbb N
\}$ is bounded. The fact that $B$ is a von Neumann algebra implies
that $\{g_b^{\circ n}(w)\}_{n\in\mathbb N}$ must have a w-convergent subsequence. Since $\Im g_b^{\circ n}(w)>\mathbb E\left[
(\Im(bx))^{-1}\right]^{-1}$, it is clear that any such limit
point must belong to $\mathbb H^+(B)$. However, as noted before,
$\lim_{n\to\infty}\varphi(g_{tb+(1-t)b_1}^{\circ n}(w))$
exists for all $t\in V$, $\varphi\in B^*$. Thus, there can only
be one limit point, i.e. $\lim_{n\to\infty}g_{b}^{\circ n}(w)
=\omega_2(b)$ must exist for all $b$ with $\Im(bx)>0$.

Note that in fact we have shown more: the limit function
$\omega_2(b):=\lim_{n\to\infty}
g_b^{\circ n}(w)$ is G\^{a}teaux holomorphic \cite[Definition 2.1]{Dineen} on all of the set $\{b\in B\colon\Im(bx)>0\}$.
It is known to be holomorphic close to the origin of $B$,
where $\eta_y(\omega_2(b))=\eta_{xy}(b)$. Using the 
same convexity trick as in the previous paragraph and
the identity principle for the usual (scalar) analytic
functions, we conclude that 
$$
\eta_y(\omega_2(b))=\eta_{xy}(b)\text{ and }g_b(\omega_2(b))=
\omega_2(b),\quad \Im(bx)>0.
$$
Thus, we have proved parts (1)--(3) of our theorem. The property
of $\omega_1$ is trivial.
\end{proof}

The above theorem, as stated, has the inconvenience that it
does not cover the posssible case of non-invertible 
positive $x$ and the case of non-invertible 
$\mathbb E[y]$. However, when $B$ is finite dimensional
(a matrix algebra), we can use normality of some families of analytic 
maps from $\mathbb C^N$ to itself to considerably improve
our result. We shall use the same notations as in Theorem \ref{Main-op}.

\begin{prop}\label{main-finite}
Let $B$ be finite-dimensional.
For any $x\ge0$, $y=y^*$ free over $B$, there exists a
domain $\mathcal D\subset B$ containing $\mathbb C^+\cdot 1$ and
an analytic map
$\omega_2\colon\mathcal D\to\mathbb H^+(B)$
so that 
$$
\eta_y(\omega_2(b))=\eta_{xy}(b)\text{ and }g_b(\omega_2(b))=
\omega_2(b),\quad b\in\mathcal D.
$$
Moreover, $\omega_2(b)=\lim_{n\to\infty}g_b^{\circ n}(w)$ for any
$w\in\mathbb H^+(B)$, $b\in\mathcal D$.
\end{prop}
\begin{proof}
We shall take as $\mathcal D=\text{int}\{b\in B\colon\Im(bx)\ge0,b\text{ invertible in }B\}$. 
By Remark \ref{pos},
$g_b\colon\mathbb H^+(B)\to\mathbb H^+(B)$ is well-defined and analytic
for any $b$ with $\Im(bx)\ge0$. The existence of 
an attracting fixed point for $g_b$ when $b$ is very close to
the origin follows by exactly the same argument as in the
proof of Theorem \ref{Main-op}. 
Since the family $\{\mathcal D\ni b\mapsto g_b^{\circ n}(w)\}_{n\in
\mathbb N}$ is normal, we conclude as before that
$\omega_2(b):=\lim_{n\to\infty}g_b^{\circ n}(w)$ exists
and is analytic. 
(Here the fact that $\text{dim}(B)<\infty$
is essential!) By the identity principle, it follows that
$\omega_2(b)=g_b(\omega_2(b))$ for all $b\in\mathcal D$, as
the relation is known to hold for $b$ of small norm.

Up to this point, we have shown the existence of an $\omega_2$
defined on the open set $\mathcal D\supset\mathbb C^+\cdot1$
which satisfies $g_b(\omega_2(b))=\omega_2(b)$, independently
from the invertibility of either $x$ or $\mathbb E[y]$. Since
$\text{dim}(B)<\infty$, the spectrum of $\mathbb E[y]$ is a finite
set in $\mathbb R$, so it follows that non-invertibility of 
$\mathbb E[y]$ is equivalent to $\mathbb E[y]\in B$ having
zero as eigenvalue. For any $\varepsilon>0$, $\mathbb E[y+\varepsilon
\cdot1]=\mathbb E[y]+\varepsilon
\cdot1$ is then invertible in $B$, and so is $x+\varepsilon\cdot1$
in $\mathcal M$. If $g_{b,\varepsilon}(w)=
h_{b(x+\varepsilon)}(h_{y+\varepsilon}(w))$, $w\in\mathbb H^+(B),
\Im(b(x+\varepsilon))>0$, it follows that
$g_{b,\varepsilon}\to g_b$ uniformly on compact subsets of
$\mathbb H^+(B)$ as $\varepsilon\to0$. Since for $b$ of small
norm, $\omega_2(b)$ is an attracting fixed point for $g_b$,
it follows that the small attracting fixed points of 
$g_{b,\varepsilon}$, which we shall call $\omega_2^\varepsilon(b)$,
converge to $\omega_2(b)$. Normality allows us to conclude that
$\omega_2^\varepsilon\to\omega_2$. 

On the other hand, as seen in Theorem \ref{Main-op}, 
$$
\eta_{y+\varepsilon}\circ\omega_2^\varepsilon=\eta_{(x+\varepsilon)
(y+\varepsilon)},\quad\varepsilon>0.
$$
Since $\eta_{y+\varepsilon}\to\eta_y$ and $\eta_{(x+\varepsilon)
(y+\varepsilon)}\to\eta_{xy}$, we conclude that
$$
\eta_y\circ\omega_2=\eta_{xy}\quad\text{on }\mathcal D.
$$
\end{proof}

\section{Numerical Implementation} 

We now consider some of the practical implications of Theorem \ref{Main-op} in the computation of the free multiplicative convolution of operator-valued random variables - in particular, those which are represented by $n\times n$ matrices. 
The general frame for those examples will be the following. Our problems will be given (directly or after some manipulations) by two matrices
$x=(x_{ij})$ and $y=(y_{ij})$ where the entries of those matrices are living
in some non-commutative probability space $(\cA,\tau)$, and where, with respect to $\tau$, the entries of $x$ are free from the entries of $y$. 

Let us set 
$\cM:=M_n(\cA)=M_n(\CC)\otimes \cA$
and consider the trace
$\varphi:=\tr\otimes \tau: \cM\to\CC$
and the conditional expectation
$\EE:=\id\otimes \tau: \cM\to M_n(\CC)$.

Thus $x$ and $y$ are elements in the non-commutative probability space
$(\cM,\varphi)$ and we are actually interested in the scalar-valued distribution of $z:=xy$ (or some variant thereof) with respect to $\varphi=\tr\otimes\tau$. 
However, the freeness between the entries of $x$ and the entries of $y$ with respect to $\tau$ does in general not ensure freeness between $x$ and $y$
with respect to $\tr\otimes\tau$. What it implies is operator-valued freeness
between $x$ and $y$ in the non-commutative operator-valued probability
space $(\cM,\EE,M_n(\CC))$. Thus we can first calculate the operator-valued distribution of $z$ with respect to $\EE=\id\otimes \tau$
by using our free convolution results and then get the scalar-valued distribution of $z$ with respect to $\varphi=\tr\otimes\tau$ by applying the trace.
More specifically, if $G_z$ is the $M_n(\CC)$-valued Cauchy transform of $z$,
then $\tr(G_z)$ is the scalar-valued Cauchy-transform of $z$. Thus the
spectral distribution of $z$ with respect to $\tr\otimes\tau$ is given, by virtue of
the Cauchy-Stieltjes inversion formula, by:
\begin{equation} \label{cau_stil} 
d\mu(t) = \lim_{\epsilon \rightarrow 0^{+}} \frac{-1}{\pi} \Im \left( \text{tr} (G_z( (t+i\epsilon)I_n)) \right), \ t\in \mathbb{R} \end{equation}
where $I_n$ is the $n\times n$ identity matrix and $\text{tr}:=\frac 1n\text{Tr}$ is the normalized trace on $n\times n$ matrices. So our objective requires that we compute the Cauchy transform of $xy$ at points of the form $zI_n$, where $z\in \mathbb{C}$.  


Note now that by (\ref{F}) and (\ref{h}), we have: 

\begin{equation} G_{xy}(zI_n) = ( zI_n - h_{xy}(z^{-1}I_n))^{-1} \end{equation} 

\noindent so computing $G_{xy}$ is equivalent to computing $h_{xy}$. Moreover, by Theorem \ref{Main-op} and (\ref{eta}) we see that 

\begin{equation} 
zh_{xy}(zI_n)=\omega_2(zI_n)h_y(\omega_2(zI_n)) 
\end{equation}

\noindent Theorem \ref{Main-op} expresses the function $\omega_2$ as the limit of iteratively composing the function $g_b$ with itself. (Since in our case $B=M_n(\CC)$ is finite-dimenisonal, we are actually in the realm of Prop.~\ref{main-finite}, and thus do not have to bother about the invertibility assumptions.) It is now clear that given
the operator-valued $h$ transform of $x$ and $y$, we can numerically compute the $h_{xy}$, and thus the spectral distribution of the product $xy$. We consider several concrete examples below, in Sections \ref{semi_ex} and \ref{dis_ex}. 


One obstacle to applying this technique to general problems is the difficulty in analytically computing the $h$ transform of many elements. It is easy to find the exact expression for the $h$ transform of discrete distributions (see Section \ref{dis_ex}), and numerical methods exist to compute the $h$ transform in other cases (see Section \ref{semi_ex} for example). 

\subsection{The Product of Two Free Operator-Valued Semicirculars}\label{semi_ex}

Let $s_1$, $s_2$, $s_3$, and $s_4$ be free, semi-circular random variables,
in some scalar-valued non-commutative probability space $(\cA,\tau)$. 
Consider the matrices $S_1$ and $S_2$ defined by:
\begin{equation} S_1 = \begin{pmatrix}
  s_1 & s_1  \\
  s_1 & s_2  
 \end{pmatrix}, \text{     }S_2 = \begin{pmatrix}
  s_3+s_4 & 2s_4  \\
  2s_4 & s_3-3s_4 
 \end{pmatrix}
\end{equation}

Matrices $S_1$ and $S_2$ represent limits of random matrices, where $s_1,\dots,s_4$ are replaced by independent Gaussian random matrices.

As before, we set 
$\cM:=M_2(\cA)=M_2(\CC)\otimes \cA$
and consider the trace
$\varphi:=\tr\otimes \tau: \cM\to\CC$
and the conditional expectation
$\EE:=\tau\otimes\id: \cM\to M_2(\CC)$.

We wish to compute the spectral distribution of $(S_2+cI_2)S_1$ in the scalar-valued probability space $(\cM,\varphi)$, where $c$ is some constant chosen large enough to make $S_2+cI_2$ positive. 
Since $S_1$ and $S_2$ are not free in $(\cM,\varphi)$ we cannot invoke usual free probability theory to achieve our goal. However,
$S_1$ and $S_2$ are free operator-valued semicircular elements in the 
operator-valued probability space $(\cM,\EE,M_2(\CC))$. Thus we can do the calculations on the operator-valued level and in the end go down to the scalar-valued level by taking the trace.

The first task is to compute
 the operator-valued $h$ transforms, $h_{S_1}$ and $h_{S_2}$. In this case, 
an analytic equation for the $h$ transforms is difficult to achieve. However, we can compute these $h$ transforms numerically using the method described in \cite{HRS07}. In brief, this involves expressing the Cauchy transform of the operator-valued semicircular in terms of the fixed point of a contraction mapping. Specifically, if we define
\begin{equation} W(b) = \lim_{n \rightarrow \infty} \mathcal{F}^{\circ n}_{b}(W_0) \end{equation}
where $\mathcal{F}_b(W) = \left( -ib + \mathbb{E}[SbS] \right)^{-1}$, then $G_{S}(b) = -iW(b)$. Note that we require the initial state $W_0$ to satisfy $\Im(W_0)>0$; convergence of the above iteration scheme is ensured by arguments from \cite{HRS07}. In our case, with
$b=(b_{ij})_{i,j=1}^2$,
we have

\begin{equation*}  \mathbb{E}[S_1bS_1] = \begin{pmatrix} b_{11} + b_{12}+b_{21}+b_{22} & b_{11}+ b_{21} \\ b_{11}+b_{12} & b_{11}+b_{22} \end{pmatrix}\end{equation*}
and
\begin{equation*}   \mathbb{E}[S_2bS_2] = \begin{pmatrix} 2b_{11} + 2b_{21}+2b_{12} + 4b_{22} & 2b_{11}+ -2b_{12} +4b_{21} -6b_{22} \\ 2b_{11} +4b_{12}-2b_{21} -6b_{22} & 4b_{11} -6b_{12} -6b_{21} +10b_{22} \end{pmatrix}\end{equation*}

We compare the spectral distribution of $S_1$ and $S_2$ computed using this method and the Cauchy-Stieltjes inversion formula to random matrix simulations in Fig. \ref{c}.

Finally, using the numerically computed $h$ transforms of $S_1$ and $S_2+cI_2$ we used the iterative method discussed here to compute the $h$ transform of their product. In Figure \ref{c}, we compare the distribution computed using our method to random matrix simulations of the ground truth spectral distribution of 
$\sqrt{S_2+cI_2}S_1\sqrt{S_2+cI_2}$. 

\begin{figure}[h]
\begin{center}
\includegraphics[scale=0.13]{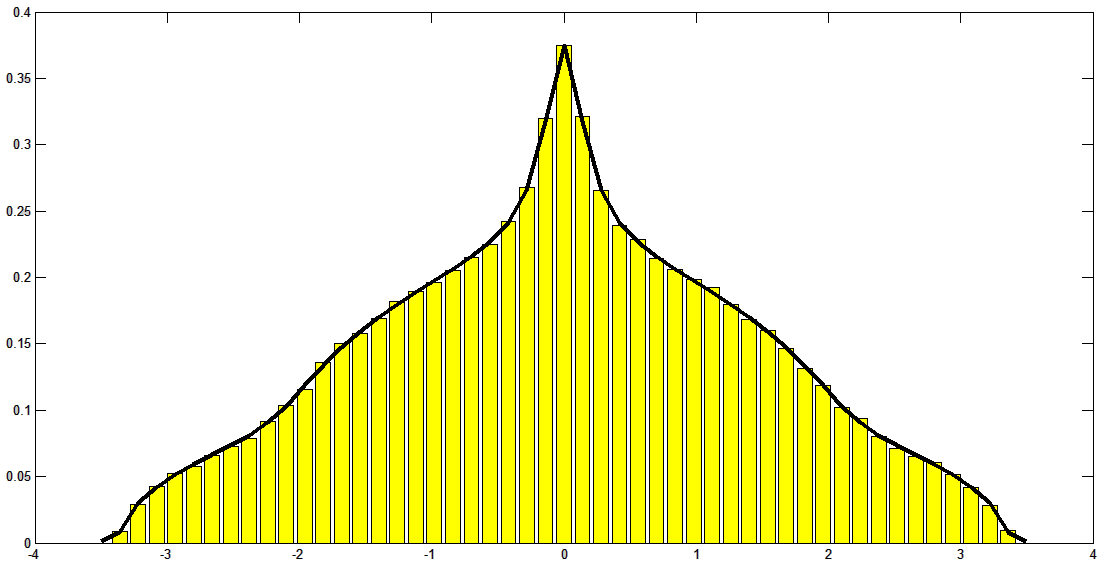}\quad
\includegraphics[scale=0.13]{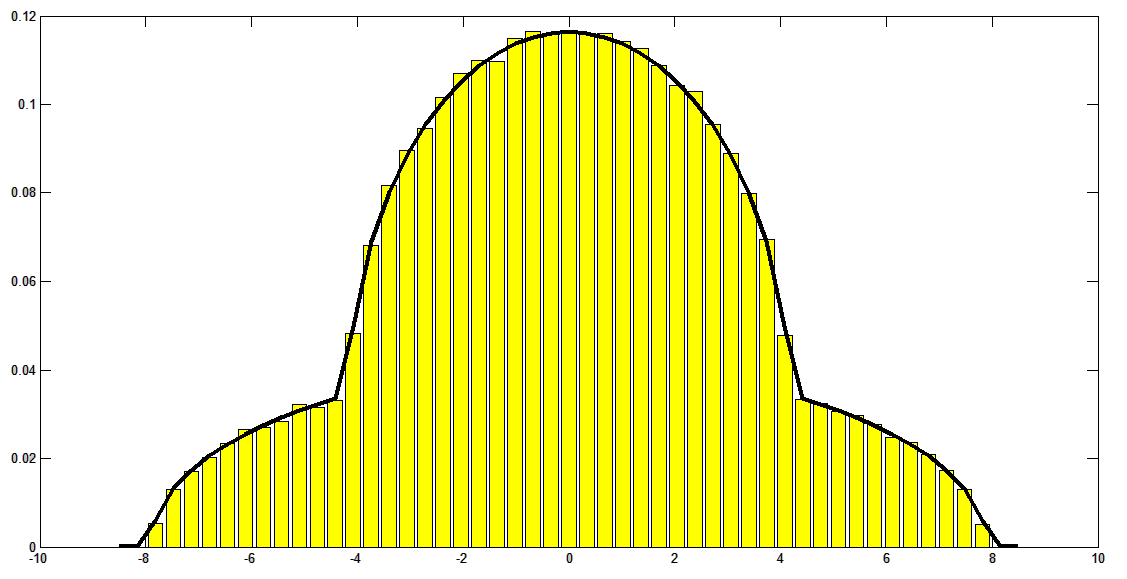}\quad
\includegraphics[scale=0.13]{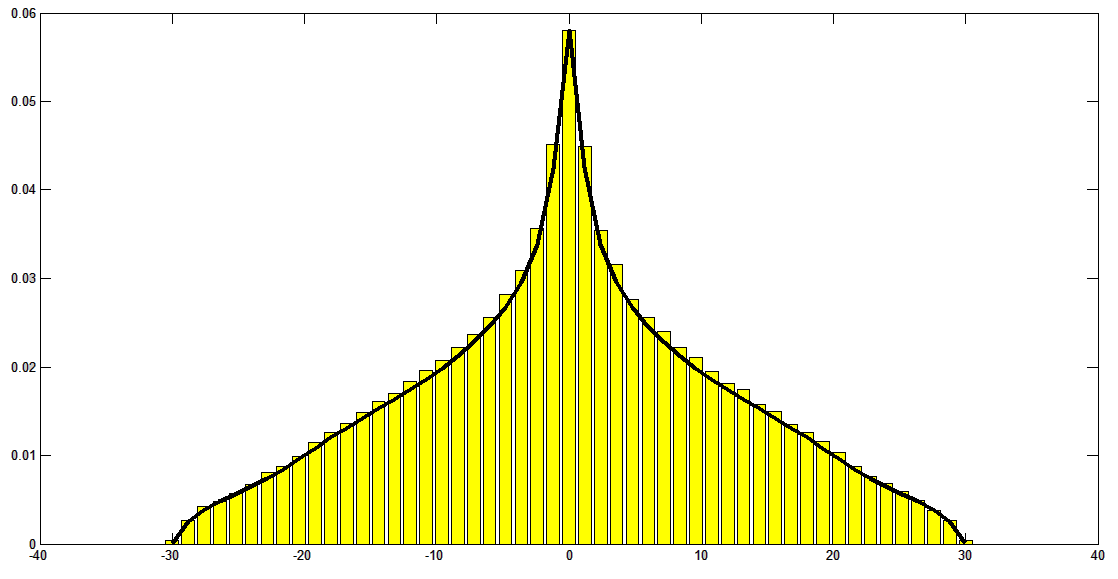}
\end{center}
\caption{Spectral distribution of $S_1$ (left), $S_2$ (middle), and $(S_2 + 8.5I_2)S_1$ (right) - random matrix simulations (histogram) compared with numerically calculated density, using fixed point method of \cite{HRS07} for $S_1$ and $S_2$ and using our method for $(S_2 + 8.5I_2)S_1$ .}
\label{c}
\end{figure}

For the sake of variety, we consider another operator-valued semi-circular example. Let now $\{s_i \}_{i=1}^6$ be free semi-circular elements, and let:

\begin{equation} S_{1}^{\prime} = \begin{pmatrix}
  -10s_1 & 2s_2 & 30s_3  \\
  2s_2 & -4s_3 & 5s_1 \\
  30s_3 & 5s_1 & 16s_1  
 \end{pmatrix} \text{ and } S_2^{\prime} = \begin{pmatrix}
  -2s_4 + 3s_6 & 3s_5 + 30s_6 & s_6 \\
  3s_5 + 30s_6 & s_4+s_5+s_6 & s_4 \\
  s_6 & s_4 & 40s_4  
 \end{pmatrix} \end{equation}

We follow the same pattern as previously: applying the numerical method proposed in \cite{HRS07} to compute the individual $h$ transforms of $S_1^{\prime}$ and $S_2^{\prime}$ (see Figure \ref{s1z}), and then using our iterative method to compute the spectral distributions of $(S_2^{\prime} + 85I_3)(S_1^{\prime} + 40I_3)$ and $(S_2^{\prime} + 85I_3)(S_1^{\prime} + 75I_3)$ (see Figure \ref{joint1z}). 

\begin{figure}[h]
\begin{center}
\includegraphics[scale=0.2]{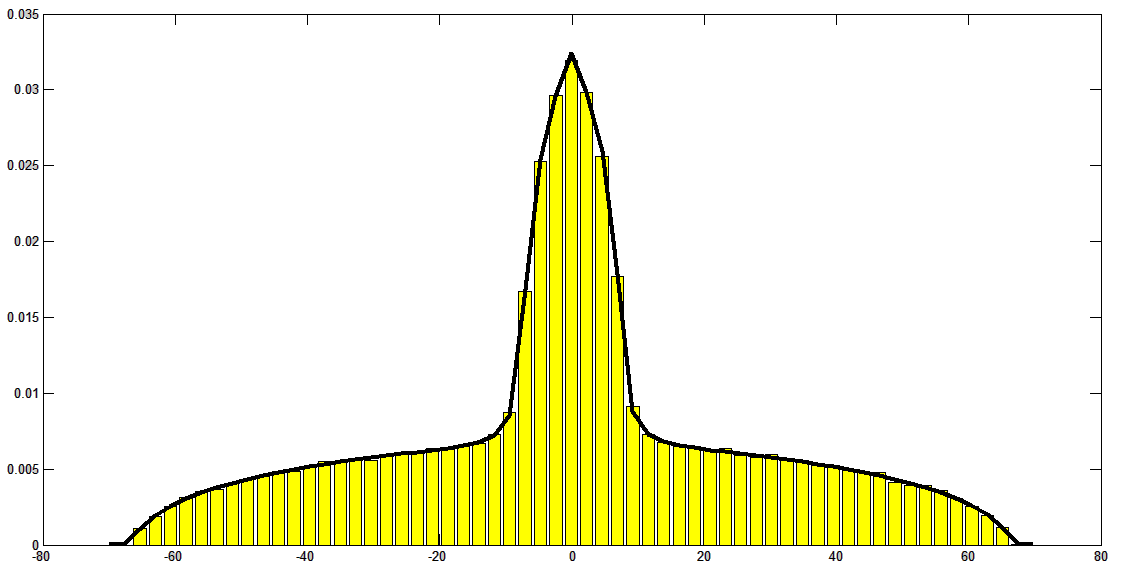}\qquad
\includegraphics[scale=0.2]{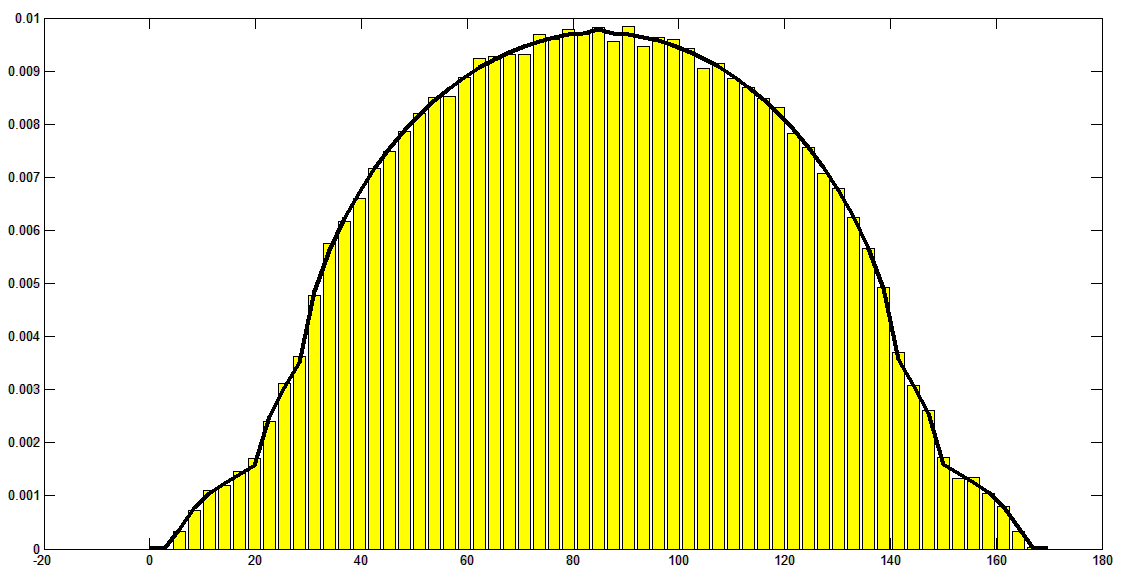}
\end{center}
\caption{Spectral distribution of $S_1^{\prime}$ (left) and $S_2'$ (right) - random matrix simulations (histogram) compared with numerically calculated density using fixed point method of \cite{HRS07}.}
\label{s1z}
\end{figure}

\begin{figure}[h]
\begin{center}
\includegraphics[scale=0.2]{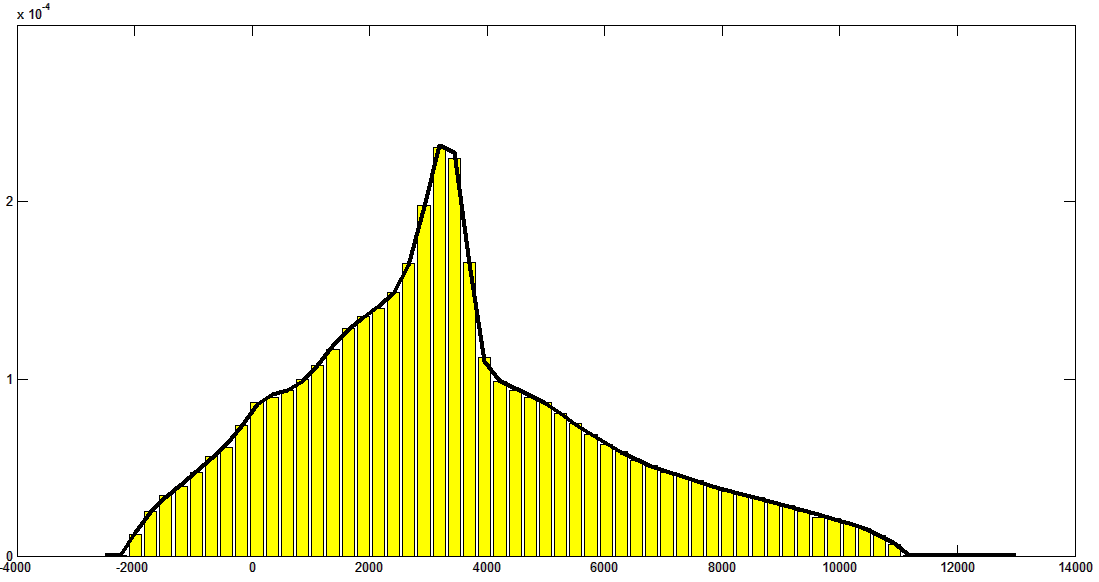}\qquad
\includegraphics[scale=0.2]{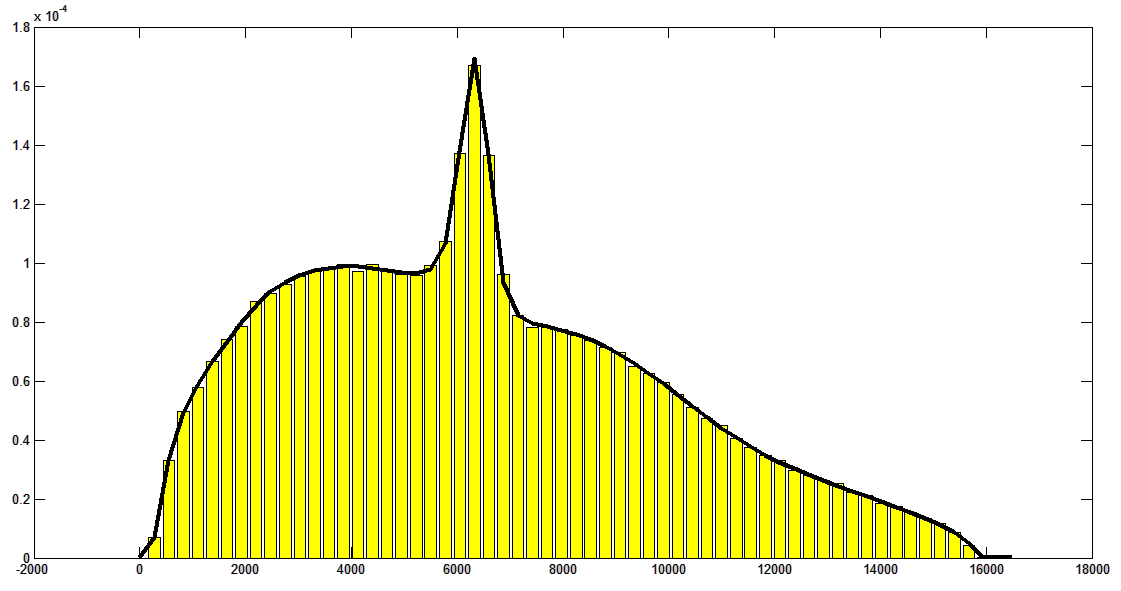}
\end{center}
\caption{Spectral distribution of $(S_2^{\prime} + 85I_2)(S_1^{\prime}+40I_3)$ (left) and $(S_2^{\prime} + 85I_2)(S_1^{\prime}+75I_3)$ (right)
- random matrix simulations (histogram) compared with numerically calculated density using our method.}
\label{joint1z}
\end{figure}





\subsection{Distribution of $dcd+d^2cd^2$}\label{dis_ex}

We consider now a special case of the problem from the Introduction, namely of finding the distribution of $dcd + d^2cd^2$, where $c$ and $d$ are free from one another. As noted in the Introduction

\begin{equation} \begin{pmatrix}
  dcd + d^2cd^2 & 0  \\
  0 & 0  
 \end{pmatrix} = \begin{pmatrix}
  d & d^2  \\
  0 & 0
 \end{pmatrix} \begin{pmatrix}
  c & 0  \\
  0 & c
 \end{pmatrix}\begin{pmatrix}
  d & 0  \\
  d^2 & 0
 \end{pmatrix}
\end{equation}
has the same distribution as

\begin{equation}  \begin{pmatrix}
  c & 0  \\
  0 & c
 \end{pmatrix}\begin{pmatrix}
  d^2 & d^3  \\
  d^3 & d^4
 \end{pmatrix} 
\end{equation}

\noindent Now, since $c$ and $d$ are free, the problem of calculating the distribution of $dcd + d^2cd^2$ has been transformed into one which can be solved numerically 
using the iterative method proposed in this paper. It is crucial to note that since $c$ and $d$ are free, we have that $\begin{pmatrix}
  d^2 & d^3  \\
  d^3 & d^4
 \end{pmatrix}$ and $\begin{pmatrix}
  c & 0  \\
  0 & c
 \end{pmatrix}$ are free over the matrices $M_2(\mathbb{C})$.

In order to apply our iterative method from Section 2, we need as input the
operator-valued Cauchy transforms (or the $h$ transforms) of the matrices
$x$ and $y$, where 
\begin{equation}   
x = \begin{pmatrix}
  c & 0  \\
  0 & c
 \end{pmatrix} \text{ and } y = \begin{pmatrix}
  d^2 & d^3  \\
  d^3 & d^4
 \end{pmatrix}. \end{equation}
For instance, we have that
\begin{equation} G_x(b) = \mathbb{E} \left[ (b-x)^{-1} \right] = \EE\left[\frac{1}{  (b_{11}-c)(b_{22}-c) - b_{12}b_{21} }  \begin{pmatrix}
  b_{22}-c & -b_{12}  \\
  -b_{21}  & b_{11}-c
 \end{pmatrix}\right] \end{equation}
and a similar formula for $G_y$.
(Note that the entries of the $2\times 2$ matrix $b-x$ commute and thus the usual formula for matrix inversion applies.) Thus we need to be able to
calculate quantities like
$$\tau\left[ [(b_{11}-c)(b_{22}-c) - b_{12}b_{21}]^{-1}(b_{22}-c)\right]$$
in order to calculate $G_x(b)$.  

For instance, if we assume here that $c$ and $d$ are both discretely distributed, such expressions can be readily
written down in analytic forms.
Consider the concrete example where $c$ is uniformly distributed with discrete support $ \{0.4, 0.7, 1, 1.3, 1.5, 1.7 \}$ and $d$ is uniformly distributed with discrete 
support $\{ 0.5, 1, 1.5, 2, 2.5, 3 \}$. In this case, the distribution of $dcd + d^2cd^2$ computed using the iterative method proposed in this paper is shown and compared to histograms in Figure \ref{t}.  

We also compute the distribution of $dcd + d^2cd^2$ where $d$ is uniformly distributed over support $ \{0.4, 0.7, 1, 1.3, 1.5, 1.7 \}$ and $c$ is a shifted 
semi-circular element (Figure \ref{t}).  

\begin{figure}[h]
\begin{center}
\includegraphics[scale=0.2]{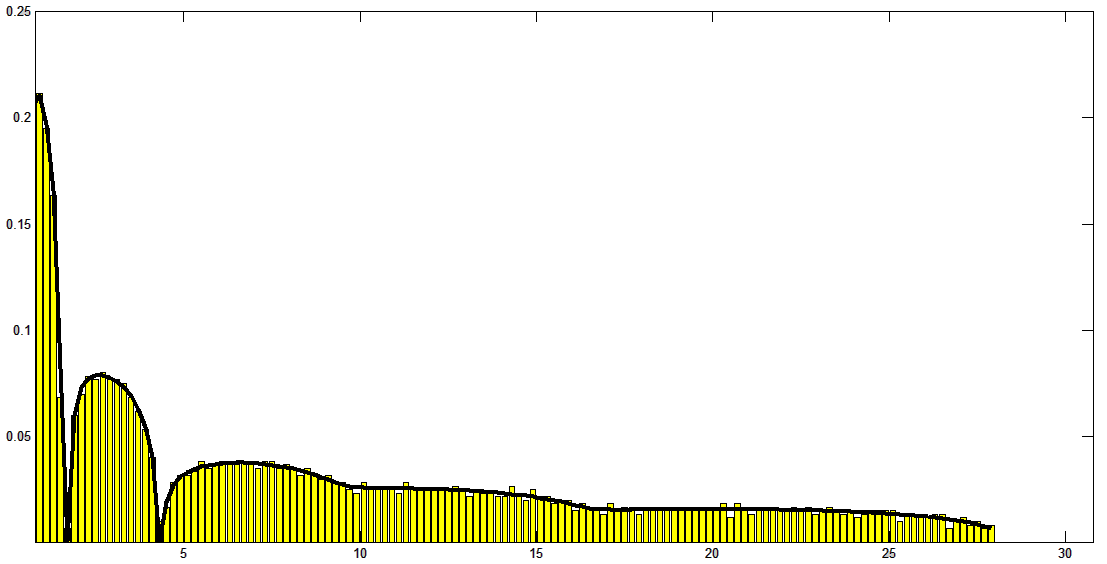}\qquad
\includegraphics[scale=0.2]{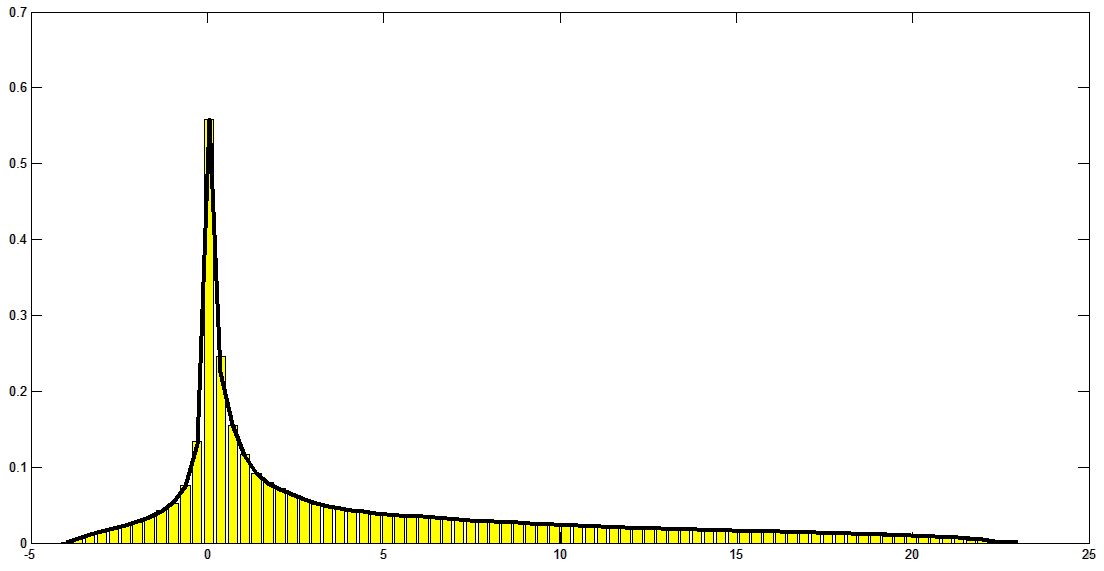}
\end{center}
\caption{Comparison of distribution of $dcd + d^2cd^2$ where c is uniformly distributed over $ \{0.4, 0.7, 1, 1.3, 1.5, 1.7 \}$ and d is uniformly distributed over 
  $\{ 0.5, 1, 1.5, 2, 2.5, 3 \}$ (left) and where c is a shifted semi-circle and d is uniformly distributed over 
  $\{ 0.4, 0.7, 1, 1.3, 1.5, 1.7 \}$ (right).}
\label{t}
\end{figure}

\end{document}